\documentclass[reqno]{amsart}
\usepackage{amsmath, amssymb, amsthm, geometry, enumerate, graphicx, amsfonts, hyperref, mathrsfs}
\hypersetup{colorlinks=true,linkcolor=red, anchorcolor=blue, citecolor=blue, urlcolor=red, filecolor=magenta, pdftoolbar=true}
\theoremstyle{plain}
\newtheorem{thm}{\it Theorem}[section]

\newtheorem{lem}[thm]{\it Lemma}
\theoremstyle{remark}
\newtheorem{defn}[thm]{Def{}inition}
\newtheorem{rem}[thm]{Remark}
\newtheorem{exa}[thm]{Example}
\numberwithin{equation}{section}
\usepackage{enumerate}
\usepackage{xcolor}
\allowdisplaybreaks
\begin{document}
\title [Nagy and Pollard-Hilding-Type Stability of Hilbert-Schmidt Frames]{Nagy and Pollard-Hilding-Type Stability of Hilbert-Schmidt Frames}

\author[Jyoti]{Jyoti}
	\address{{\bf{Jyoti}}, Department of Mathematics,
		University of Delhi, Delhi-110007, India.}
	\email{jyoti.sheoran3@gmail.com}

\author[Ruchi]{Ruchi}
	\address{{\bf{Jyoti}}, Department of Mathematics,
		University of Delhi, Delhi-110007, India.}
	\email{rgarg@maths.du.ac.in}

\author[Lalit   Kumar Vashisht]{Lalit  Kumar  Vashisht$^{*}$}
	\address{{\bf{Lalit  Kumar  Vashisht}}, Department of Mathematics,
		University of Delhi, Delhi-110007, India.}
	\email{lalitkvashisht@gmail.com}

\begin{abstract}
We establish several sufficient conditions ensuring that a perturbed family of operators in a given Hilbert-Schmidt frame for a separable Hilbert space remains a Hilbert-Schmidt frame. The perturbation criteria are formulated through  Nagy-type and Pollard-Hilding-type inequalities, yielding explicit estimates for the resulting frame bounds.   Several examples are presented to illustrate the applicability of the results.
\end{abstract}


\keywords{Bessel sequence; Frames; Compact opertor, Perturbation.\\
The research of Jyoti is supported by the  WISE-PDF research grant of WISE--KIRAN Division, Department of Science and Technology (DST), Government of India (Grant No.: DST/WISE-PDF/PM-6/2023(G)). Ruchi is supported by the National Board for Higher Mathematics (NBHM), Grant No.: 0203/23/2024/ R $\&$ D-II/647. Lalit  Kumar Vashisht is  supported by the Faculty Research Programme Grant-IoE, University of Delhi \ (Grant No.: Ref. No./IoE/2025-26/12/FRP).\\
$^*$Corresponding author}

\maketitle

\baselineskip15pt

\section{Introduction}
Frames in Hilbert spaces were originally motivated by the work of Gabor \cite{G46} on signal decomposition and were formally introduced by Duffin and Schaeffer \cite{DS} in connection with nonharmonic Fourier series. Since then, frame theory has developed into an active area of research with applications in harmonic analysis, signal processing, sampling theory, and operator theory; see \cite{BHan,H11,Young,Zpelle}. Let $\mathcal{H}$ be a separable Hilbert space. A sequence $\{x_i\}_{i \in I}\subseteq \mathcal{H}$ is called a frame for $\mathcal{H}$ if there exist positive real numbers $a$ and $b$  such that
\begin{align}\label{b01.1}
a \|x\|^2\leq  \sum_{i \in I}|\langle x, x_i\rangle|^2 \leq b \|x\|^2 \ \text{for all} \ x \in \mathcal{H}.
\end{align}
The numbers $a$ and $b$ are referred to as the lower and upper frame bounds, respectively. If only the upper inequality in \eqref{b01.1} is satisfied, then $\{x_i\}_{i \in I}$   is  called a Bessel sequence with Bessel bound $b$. For a frame
$\{x_i\}_{i \in I}$, the operator $S: \mathcal{H} \rightarrow \mathcal{H}$ defined by $Sx = \sum_{i \in I}\langle x, x_i\rangle x_i$ is called the frame operator. It is well known that $S$ is bounded, positive, and invertible. Consequently, every $x\in\mathcal{H}$ admits the reconstruction formula: $x = SS^{-1}x =\sum\limits_{i \in I} \langle x, S^{-1}x_i \rangle x_i$.

\subsection{Related work}
Several generalizations of classical frames have been introduced over the years. In particular, Sun \cite{WsunI} introduced the notion of $g$-frames by considering families of bounded linear operators acting from a Hilbert space into closed subspaces of another Hilbert space. More precisely, let $\{K_i\}_{i \in I}$ be a family of closed subspaces of a separable Hilbert space $\mathcal{K}$. A collection $\{L_i\}_{i \in I}$ of bounded linear operators from $\mathcal{H}$ into $K_i$ is called a $g$-frame for $\mathcal{H}$ with respect to $\{K_i\}_{i \in I}$ if there exist constants $\alpha,\beta>0$ such that
\begin{align*}
\alpha \|x\|^2 \leq \sum\limits_{i \in I}\|L_i x\|^2 \leq \beta \|x\|^2, \  x \in \mathcal{H}.
\end{align*}
Perturbation results for $g$-frames were also established by Sun in \cite{WsunII}. The interaction between frame theory and operator theory has led to the study of various operator-valued frame systems. Koo and Lim \cite{Yoo} investigated Schatten $p$-class operators through frame-theoretic techniques, while von Neumann-Schatten $p$-frames in Banach spaces were studied in \cite{SA}. Relations between Hilbert-Schmidt operators and frames were discussed in \cite{Balaza}. Standard references for the theory of Hilbert-Schmidt operators include the books by Schatten \cite{SCHI} and Simon \cite{B.simon}.

Hilbert-Schmidt frames, consisting of bounded linear operators from a Hilbert space $\mathcal{H}$ into the Hilbert-Schmidt class $\mathcal{C}_2(\mathcal{K})$, provide a natural operator-valued extension of classical frames. These frames are closely related to $g$-frames and matrix-valued frame systems and have been studied from several viewpoints. Hilbert-Schmidt frames for the class $\mathcal{C}_2(\mathcal{K})$ were investigated in \cite{J}, while Hilbert-Schmidt frames and Hilbert-Schmidt Riesz bases associated with tensor products of Hilbert spaces were studied in \cite{JVPDF1}. Hilbert-Schmidt frames for separable Hilbert spaces, where the lower frame condition is controlled by a bounded linear operator were studied in \cite{JV23}. Duality properties for Hilbert-Schmidt frames can be seen in \cite{WZhang}.

The stability of frame systems under perturbations has attracted significant attention due to its importance in applications; see \cite{GossonI, GossonII, H11, Pert1,  Lkhari, Young}. Perturbation theory for frames plays a central role in signal processing \cite{BHan,H11}, quantum physics \cite{GossonI}, time-frequency analysis \cite{KG}, distributed signal processing \cite{Bemrose, DVV17, DV17ar}, and duality theory \cite{H11,Young}. Classical perturbation results originate from the Paley-Wiener stability theorem for bases in Hilbert spaces; see \cite{H11}. Various Paley-Wiener type perturbation results for discrete frames can be found in \cite{CK}, while stability conditions for Gabor and wavelet frames were established in \cite{FZUS}. More recently, perturbation problems for frames associated with Weyl-Heisenberg and affine systems were studied in \cite{DivI}. Further accounts of perturbation theory for frames are available in \cite{GossonII,H11,Young}.

Recall that a  sequence of vectors  $\{x_{n}\}_{n=1}^{\infty}$  in a Hilbert space $\mathcal{H}$ is said to be \emph{complete} if the closure of its span equals  $\mathcal{H}$, that is, $\overline{\text{span}}\{x_{n}\}_{n=1}^{\infty} = \mathcal{H}$.
In \cite{P}, Pollard, proved a Paley-Wiener-type stability result for complete sequences in Hilbert spaces.  Hilding, in \cite{H}, studied the  Paley-Wiener-type stability result for completeness of a sequence  in terms of a linear transformation on Hilbert spaces. A sequence in a Hilbert space is called a basic sequence if it is a basis for its closed linear span.  Retherford \cite{R} introduced Nagy-type and  Pollard--Hilding-type stability conditions for basic sequences and completeness of sequences in complete linear metric spaces and normed spaces. These are   generalized  Paley-Wiener-type stability conditions for completeness of sequences in Hilbert spaces. Most recently, the authors of \cite{RV} showed that the Nagy-type and  Pollard--Hilding-type stability conditions for frames have potential applications in dynamical sampling. They showed in \cite{RV} that the initial states in a homogeneous dynamical system can be stably recovered under the Nagy-type and  Pollard--Hilding-type stability conditions for frames in separable Hilbert spaces.  Recently, in \cite{JVPDF2}, the perturbation stability of Hilbert-Schmidt frames under structured modifications was investigated. The authors established sufficient conditions guaranteeing preservation of the Hilbert-Schmidt frame property under finite and infinite replacements of frame elements. However, nonlinear perturbation conditions of Nagy-type and Pollard-Hilding-type have not yet been studied in the setting of Hilbert-Schmidt frames to the best of the authors’ knowledge.

Motivated by this observation, the present paper develops perturbation results for Hilbert-Schmidt frames under nonlinear stability conditions inspired by the classical Nagy and Pollard-Hilding perturbation frameworks. Sufficient conditions are obtained to ensure preservation of the Hilbert-Schmidt frame property together with explicit estimates for the associated frame bounds. The results obtained in this paper demonstrate that the stability of Hilbert-Schmidt frames can be established under a broad class of nonlinear perturbation estimates extending beyond the classical additive perturbation conditions.

\subsection{Structure of the paper}
To make the paper self-contained, Section \ref{sec2} contains the notation, definitions, and preliminary results used throughout the paper. Section \ref{seci} is devoted to Nagy-type perturbation results for Hilbert-Schmidt frames, where Theorem \ref{gp2} gives sufficient conditions for a family of operators satisfying a Nagy-type inequality to constitute an Hilbert-Schmidt frame, together with explicit estimates for the corresponding frame bounds and Theorem \ref{th2} derives an alternative Nagy-type perturbation criterion in terms of the associated pre-frame operators. Section \ref{secp} presents Pollard-Hilding-type perturbation results for Hilbert-Schmidt frames. Following the framework developed for the Nagy-type perturbation results, Theorem \ref{gp3} and Theorem \ref{th3} establish perturbation criteria formulated in terms of Pollard-Hilding-type inequalities. The paper concludes with Theorem \ref{pq} together with its pre-frame operator analogue, Theorem \ref{pq1}, establishing perturbation criteria based on a different class of mixed-power inequalities.

\section{Preliminaries}\label{sec2}
Throughout, we shall let $I$ a countable indexing set. Let $\mathfrak{B}(X, Y)$ denote the space of bounded linear operators from a normed space $X$ into a normed space $Y$.  If $X=Y$, then $\mathfrak{B}(X, Y) = \mathfrak{B}(X)$.
Let  $\{e_i\}_{i \in I}$  be  an orthonormal basis for  a separable Hilbert space $\mathcal{K}$. We denote by $\mathcal{C}_2(\mathcal{K})$ the Hilbert-Schimdt class, which is a collection of all compact operators $T$ acting on $\mathcal{K}$ such that $\sum\limits_{i \in I} \|Te_i\|^2 < \infty$. The space $\mathcal{C}_2(\mathcal{K})$ is a Banach space with respect to norm $\|\cdot\|_2$ defined as
\begin{align*}
\|T\|_2= \Big(\text{\textbf{trace}}(T^* T)\Big)^{1/2} =\Big(\sum\limits_{i \in I} \| Te_i\|^2\Big)^{1/2}, \ T \in \mathcal{C}_2(\mathcal{K}).
\end{align*}
 The space $\mathcal{C}_2(\mathcal{K})$ is a Hilbert space with respect to the inner product $[ \cdot, \cdot]_{\textbf{tr}}$ defined as
\begin{align*}
[ T, S ]_{\textbf{tr}}= \text{\textbf{trace}}(S^* T)=\sum_{i \in I} \langle T e_i, S e_i \rangle , \ T, S \in \mathcal{C}_2(\mathcal{K}).
\end{align*}
The space $\bigoplus\limits_{i \in I} \mathcal{C}_2(\mathcal{K})= \Big\{\{A_i\}_{i \in I} \subseteq  \mathcal{C}_2(\mathcal{K}):  \Big(\sum\limits_{i \in I}\|A_i \|_2^2\Big)^{1/2}< \infty\Big\}$
 is a Hilbert space with respect to the inner product given by
\begin{align*}
\big\langle \{A_i\}_{i \in I}, \{B_i\}_{i \in I} \big \rangle = \sum\limits_{i \in I} [ A_i,B_i ]_{\textbf{tr}}.
\end{align*}
Let   $ x$,  $y \in \mathcal{K}$. The map  $x \otimes y : \mathcal{K} \rightarrow \mathcal{K}$  given by
$(x \otimes y )(u)= \langle u, y \rangle x, \ u \in \mathcal{K}$, is an element of  $\mathfrak{B}(\mathcal{K})$ with
$\|x \otimes y\|= \|x\| \|y\|$. Moreover, $x \otimes y \in \mathcal{C}_2(\mathcal{K})$ with $\|x \otimes y\|_2= \|x\|\|y\|$.
For any  $ x, y, u, v \in \mathcal{K}$ and $U \in \mathcal{B}(\mathcal{K})$,  the following properties hold:
\begin{enumerate}[$(i)$]
\item $(x \otimes y)^*= y \otimes x$.
\item $\textbf{trace}(x \otimes y)=\langle x, y \rangle$.
\item $[ x \otimes y, u \otimes v ]_{\textbf{tr}}=\langle x, u \rangle \langle v, y \rangle$.
\item $(x \otimes y) (u \otimes v)= \langle u, y \rangle (x \otimes v)$.
\item $ U (x \otimes y)= Ux \otimes y$ and $(x \otimes y)U=x \otimes U^*y$.
\end{enumerate}

\subsection{Hilbert-Schmidt frames}
Here we recall basic notions related to Hilbert-Schmidt frames.
\begin{defn}\cite{SA}
A sequence $\{G_i\}_{i \in I}$ of bounded linear operators from a Hilbert space $\mathcal{H}$ into the Hilbert-Schmidt class
$\mathcal{C}_2(\mathcal{K})$ is called a Hilbert-Schmidt frame (HS-frame) for $\mathcal{H}$ with respect to $\mathcal{K}$ if there exist constants $A, B > 0$ such that
\begin{align}\label{eq3.2mf}
A \|x\|^2 \leq \sum_{i \in I} \|G_i x\|_2^2 \leq B \|x\|^2 \quad \text{for all} \ x \in \mathcal{H}.
\end{align}
\end{defn}
The scalars $A$ and $B$ are called frame bounds of $\{G_i\}_{i \in I}$. The HS-frame $\{G_i\}_{i \in I}$ is called a Parseval HS-frame if it is possible to choose $A=B=1$.
If only the upper inequality in \eqref{eq3.2mf} is satisfied,   then we say that $\{G_i\}_{i \in I}$  is  a Hilbert-Schmidt Bessel  (HS-Bessel) sequence with Bessel bound $B$.
For a given HS-Bessel sequence $\{G_i\}_{i \in I}$ in  $\mathcal{H}$, the map  $\mathcal{T}_G: \bigoplus\limits_{i \in I} \mathcal{C}_2(\mathcal{K}) \rightarrow \mathcal{H}$ defined by
\begin{align*}
 \mathcal{T}_G \big(\{A_i \}_{i \in I} \big) =  \sum_{i \in I} G_i^* A_i,  \ \{A_i \}_{i \in I} \in \bigoplus\limits_{i \in I} \mathcal{C}_2(\mathcal{K})
\end{align*}
is called the \emph{pre-frame operator} of   $\{G_i\}_{i \in I}$.
The \emph{analysis operator} $\mathcal{T}_G^*:\mathcal{H} \rightarrow    \bigoplus\limits_{i \in I} \mathcal{C}_2(\mathcal{K})$ is given  by
\begin{align*}
 \mathcal{T}_G^*  x  =   \big\{ G_i x\big\}_{i \in I},  \ x \in \mathcal{H}.
\end{align*}
The pre-frame operator and analysis operator are  linear and bounded with $\|\mathcal{T}_G\|=\|\mathcal{T}_G^*\|\leq \sqrt{B}$. The \emph{frame operator} of  $\{G_i\}_{i \in I}$ is the composition $ \mathcal{S} = \mathcal{T}_G\mathcal{T}_G^*: \mathcal{H} \rightarrow   \mathcal{H}$ given by
\begin{align*}
 \mathcal{S}_G x =  \sum_{i \in I} G_i^* G_i x,  \ x \in \mathcal{H}.
\end{align*}
The frame operator $\mathcal{S}_G$ is bounded and linear. Further, if  $\{G_i\}_{i \in I}$ is an HS-frame, then the frame operator is positive and invertible  on $\mathcal{H}$. Thus, we have the following decomposition formulas:
\begin{align}\label{fo}
x=\sum_{i \in I} G_i^* G_i \mathcal{S}_G^{-1}x=\sum_{i \in I} \mathcal{S}_G^{-1}G_i^* G_i x, \ x \in \mathcal{H}.
\end{align}

\begin{defn}
Let $\{G_i\}_{i \in I}$ be an HS-frame for $\mathcal{H}$ with respect to  $\mathcal{K}$. An HS-frame $\{\widetilde{G}_i\}_{i \in I}$ is called a dual  of $\{G_i\}_{i \in I}$ if
\begin{align*}
\sum_{i \in I} G_i^* \widetilde{G}_i x=\sum_{i \in I} \widetilde{G}_i^* G_i x=x, \ x \in \mathcal{H}.
\end{align*}
It is well-known that if $\{G_i\}_{i \in I}$ is an HS-frame with frame bounds $A, B$, then $\{G_i\mathcal{S}_G^{-1}\}_{i \in I}$ is an HS-frame with frame bounds $B^{-1}, A^{-1}$. Moreover, $\{G_i\mathcal{S}_G^{-1}\}_{i \in I}$ a dual of $\{G_i\}_{i \in I}$ due to \eqref{fo} and is known as the canonical dual of $\{G_i\}_{i \in I}$.
\end{defn}

\subsection{Auxiliary results}

\begin{lem}\cite{CK}\label{L1}
Let $\mathcal{X}$ be a Banach space and $U : \mathcal{X} \rightarrow \mathcal{X}$ be a linear operator. Assume that there exist constants $\lambda_{1}, \lambda_{2} \in [0,1)$ such that
\begin{align*}
\|Ux-x\| \leq \lambda_{1} \|x\| + \lambda_{2} \|Ux\|, \ x \in \mathcal{X}.
\end{align*}
Then $U$ is bounded and invertible. Moreover,
\begin{align*}
\frac{1-\lambda_1}{1+ \lambda_2} \|x\| \leq \|Ux\| \leq \frac{1+\lambda_1}{1- \lambda_2} \|x\| \ \ \text{and} \ \
 \frac{1-\lambda_2}{1+ \lambda_1}\|x\| \leq \|U^{-1}x\| \leq \frac{1+\lambda_2}{1- \lambda_1} \|x\|,  \ x \in \mathcal{X}.
\end{align*}
\end{lem}

\begin{thm}\cite{CD}\label{Young} \ \textbf{Young's Inequality:}
Let $1 < p, q < \infty$ be such that $\frac{1}{p}+\frac{1}{q} = 1$. Then, for all $a, b > 0$,
$ab \leq \frac{a^p}{p}+\frac{b^q}{q}$.
In particular, for $p = q = 2$ and any $\epsilon >0$, substituting  $\sqrt{2\epsilon}a$ for $a$ and
$\frac{b}{\sqrt{2\epsilon}}$ for $b$ yields the following $\epsilon$-version of Young’s inequality:
\begin{align}\label{y2}
ab \leq \epsilon a^2+ \frac{1}{4 \epsilon} b^{2}.
\end{align}
\end{thm}
	
\begin{thm}\cite{H}\label{Theorem 2.8}
For any $x \geq 0, y \geq 0$, we have $x^{k}+ y^{k} \leq c (x+y)^{k}$, where
\begin{align*}
c=\begin{cases}
1, &   k \geq 1,\\
2^{1-k}, &  k \leq 1 .
\end{cases}
\end{align*}
\end{thm}

The following classical stability notions for sequences in Banach and Hilbert spaces are closely related to Paley-Wiener type perturbation theory and motivate the perturbation inequalities considered later in this paper.

\begin{defn}\cite{R}
Two sequences $\{x_i\}_{i=1}^\infty$ and $\{y_i\}_{i=1}^\infty$ in a normed space $\mathcal{X}$ are said to satisfy
\begin{enumerate}
\item the \emph{Pollard-Hilding stability condition} if for each positive real number $k$, there exist constants $\lambda_1,\lambda_2$ with
$0\leq \lambda_j < \min\{1,2^{1-\frac{1}{k}}\}, \ j=1,2,$ such that
\begin{align*}
\Big\|\sum_{i=1}^{n} a_i(x_i-y_i)\Big\|\leq\left(\lambda_1\Big\|\sum_{i=1}^{n} a_i x_i\Big\|^k+\lambda_2\Big\|\sum_{i=1}^{n} a_i y_i\Big\|^k\right)^{1/k}
\end{align*}
for all finite scalar sequences $\{a_i\}_{i=1}^{n}$.

\item the \emph{Nagy stability condition} if there exist constants $\lambda,\nu \in [0,1)$ and a constant $\mu \geq 0$ with
$\mu^2 \leq (1-\lambda)(1-\nu)$, such that
\begin{align*}
\Big\|\sum_{i=1}^{n} a_i(x_i-y_i)\Big\|^2\leq\lambda\Big\|\sum_{i=1}^{n} a_ix_i\Big\|^2
+\mu\Big\|\sum_{i=1}^{n} a_i x_i\Big\|\Big\|\sum_{i=1}^{n} a_i y_i\Big\|
+\nu\Big\|\sum_{i=1}^{n} a_i y_i\Big\|^2
\end{align*}
for all finite scalar sequences $\{a_i\}_{i=1}^{n}$.
\end{enumerate}
\end{defn}

In \cite{P}, Pollard established a generalized Paley-Wiener type stability result for complete sequences in separable Hilbert spaces in the special case $\mu=0$, while related variants were subsequently studied by Hilding in \cite{H}. The perturbation conditions considered in the present work may be regarded as Hilbert-Schmidt frame analogues of these classical stability notions.

\section{Nagy-Type Stability of HS-Frames}\label{seci}

In this section, we establish Nagy-type stability results for HS-frames. We begin with a perturbation theorem in which the perturbation estimate contains a mixed term involving both the original and perturbed families. By controlling this mixed term through Young's inequality, we obtain sufficient conditions ensuring that the perturbed family remains an HS-frame.

\begin{thm}\label{gp2}
Let $\{G_i\}_{i\in I}$ be an HS-frame for $\mathcal{H}$ with respect to $\mathcal{K}$ with bounds $A,B$, and let $\{H_i\}_{i\in I}$ be a family of operators from $\mathcal{H}$ into $\mathcal{C}_2(\mathcal{K})$. Suppose that there exist constants $\lambda,\mu,\nu\geq 0$ such that for every finite subset $J\subseteq I$ and $x\in\mathcal{H}$,
\begin{align}\label{gi}
\Big\|\sum_{i\in J}\big(G_i^*G_i x-H_i^*H_i x\big)\Big\|^2
\leq\lambda\Big\|\sum_{i\in J}G_i^*G_i x\Big\|^2+\mu\Big\|\sum_{i\in J}G_i^*G_i x\Big\|\Big\|\sum_{i\in J}H_i^*H_i x\Big\|+\nu\Big\|\sum_{i\in J}H_i^*H_i x\Big\|^2.
\end{align}
If there exists an $\epsilon>0$ such that $\max\left\{\lambda+\mu\epsilon, \nu+\frac{\mu}{4\epsilon}\right\}<1$,
then $\{H_i\}_{i\in I}$ is an HS-frame for $\mathcal H$ with respect to $\mathcal K$ with frame bounds
$A\left(\frac{1-\sqrt{\lambda+\mu\epsilon}}{1+\sqrt{\nu+\frac{\mu}{4\epsilon}}}\right)$
and $B\left(\frac{1+\sqrt{\lambda+\mu\epsilon}}{1-\sqrt{\nu+\frac{\mu}{4\epsilon}}}\right)$.
\end{thm}
\begin{proof}
For any finite subset $J\subseteq I$ and $x \in \mathcal{H}$, we compute
\begin{align}\label{e1}
&\Big\|\sum_{i\in J}(G_i^*G_i x-H_i^*H_i x)\Big\|^2\notag\\
& \leq \lambda\Big\|\sum_{i\in J}G_i^*G_i x\Big\|^2+\mu\Big\|\sum_{i\in J}G_i^*G_i x\Big\|\Big\|\sum_{i\in J}H_i^*H_i x\Big\|+\nu\Big\|\sum_{i\in J}H_i^*H_i x\Big\|^2\notag\\
& \leq \lambda\Big\|\sum_{i\in J}G_i^*G_i x\Big\|^2+\mu\epsilon\Big\|\sum_{i\in J}G_i^*G_i x\Big\|^2 +\frac{\mu}{4\epsilon} \Big\|\sum_{i\in J}H_i^*H_i x\Big\|^2 +\nu\Big\|\sum_{i\in J}H_i^*H_i x\Big\|^2 \quad \text{(by \eqref{y2})}\notag\\
&=\left(\sqrt{\lambda+\mu\epsilon} \ \Big\|\sum_{i\in J}G_i^*G_i x\Big\|\right)^2+\left( \sqrt{\nu+\frac{\mu}{4\epsilon}} \ \Big\|\sum_{i\in J}H_i^*H_i x\Big\|\right)^2\notag\\
& \leq \left(\sqrt{\lambda+\mu\epsilon} \ \Big\|\sum_{i\in J}G_i^*G_i x\Big\| +\sqrt{\nu+\frac{\mu}{4\epsilon}} \ \Big\|\sum_{i\in J}H_i^*H_i x\Big\|\right)^2\notag\\
\intertext{that gives}
&\Big\|\sum_{i\in J}(G_i^*G_i x-H_i^*H_i x)\Big\|  \leq \sqrt{\lambda+\mu\epsilon} \ \Big\|\sum_{i\in J}G_i^*G_i x\Big\| +\sqrt{\nu+\frac{\mu}{4\epsilon}} \ \Big\|\sum_{i\in J}H_i^*H_i x\Big\|.
\end{align}
Now
\begin{align*}
\Big\|\sum_{i\in J}H_i^*H_i x\Big\|&\leq\Big\|\sum_{i\in J}(G_i^*G_i x-H_i^*H_i x)\Big\|+\Big\|\sum_{i\in J}G_i^*G_i x\Big\|\\
&\leq\left(1+\sqrt{\lambda+\mu\epsilon}\right)\Big\|\sum_{i\in J}G_i^*G_i x\Big\|+\sqrt{\nu+\frac{\mu}{4\epsilon}} \ \Big\|\sum_{i\in J}H_i^*H_i x\Big\|\\
\intertext{and}
\Big\|\sum_{i\in J}G_i^*G_i x\Big\| &=\sup_{\|y\|=1}\Big|\Big\langle\sum_{i\in J}G_i^*G_i x,y\Big\rangle\Big|
=\sup_{\|y\|=1}\Big|\sum_{i\in J}[ G_i x,G_i y]_{\textbf{tr}}\Big|
\leq \Big(\sum_{i\in J}\|G_i x\|_2^2\Big)^{1/2}\sup_{\|y\|=1}\Big(\sum_{i\in J}\|G_i y\|_2^2\Big)^{1/2}
\intertext{gives}
&\Big\|\sum_{i\in J}H_i^*H_i x\Big\|\leq B\left(\frac{1+\sqrt{\lambda+\mu\epsilon}}{1-\sqrt{\nu+\frac{\mu}{4\epsilon}}}\right)\|x\|, \ x \in \mathcal{H}.
\end{align*}
Therefore, $\sum\limits_{i\in I}H_i^*H_i x$ is unconditionally convergent for every $x\in\mathcal{H}$. Define $\mathcal{S}_H: \mathcal{H} \to \mathcal{H}$ as
\begin{align*}
\mathcal{S}_Hx=\sum_{i\in I}H_i^*H_i x,\ x\in \mathcal{H}.
\end{align*}
Then, $\mathcal{S}_H$ is a bounded operator. Moreover, for every $x \in \mathcal{H}$,
\begin{align*}
\sum_{i\in I}\|H_i x\|_2^2=\langle \mathcal{S}_H x,x\rangle\leq\|\mathcal{S}_H\|\|x\|^2 \leq B\left(\frac{1+\sqrt{\lambda+\mu\epsilon}}{1-\sqrt{\nu+\frac{\mu}{4\epsilon}}}\right) \|x\|^2,
\end{align*}
showing that $\{H_i\}_{i\in I}$ is an HS-Bessel sequence. Let $\mathcal{S}_G$ be the frame operator of $\{G_i\}_{i\in I}$. Passing to the limit over finite subsets $J \subseteq I$ in \eqref{e1}, we obtain
\begin{align*}
\|\mathcal{S}_Gx-\mathcal{S}_Hx\|\leq\sqrt{\lambda+\mu\epsilon} \ \|\mathcal{S}_Gx\|+\sqrt{\nu+\frac{\mu}{4\epsilon}}\ \|\mathcal{S}_Hx\|,\ x\in\mathcal H.
\end{align*}
Replacing $x$ by $\mathcal{S}_G^{-1}x$, we obtain
\begin{align*}
\|x-\mathcal{S}_H \mathcal{S}_G^{-1}x\|&\leq\sqrt{\lambda+\mu\epsilon} \ \|x\|+\sqrt{\nu+\frac{\mu}{4\epsilon}} \ \|\mathcal{S}_H \mathcal{S}_G^{-1}x\|.
\end{align*}
Since $\max\left\{\sqrt{\lambda+\mu\epsilon},\sqrt{\nu+\frac{\mu}{4\epsilon}}\right\}<1$, by Lemma~\ref{L1}, $\mathcal{S}_H \mathcal{S}_G^{-1}$ is invertible and therefore $\mathcal{S}_H$ is invertible and hence $\{H_i\}_{i\in I}$ is an HS-frame. Moreover,
\begin{align*}
\|\mathcal{S}_H^{-1}\|\leq\|\mathcal{S}_G^{-1}\|\|\mathcal{S}_G \mathcal{S}_H^{-1}\|\leq\frac{1+\sqrt{\nu+\frac{\mu}{4\epsilon}}}{A\left(1-\sqrt{\lambda+\mu\epsilon}\right)}
\intertext{and therefore}
\|\mathcal{S}_H^{-1}\|^{-1} \geq A\left(\frac{1-\sqrt{\lambda+\mu\epsilon}}{1+\sqrt{\nu+\frac{\mu}{4\epsilon}}}\right)
\end{align*}
gives the desired lower frame bound for $\{H_i\}_{i\in I}$. The proof is complete.
\end{proof}

\begin{rem}
The presence of the free parameter $\epsilon>0$ in Theorem \ref{gp2} allows some flexibility in the resulting frame bounds. By choosing $\epsilon$ appropriately, one can obtain the sharpest guaranteed frame bounds. Specifically, this is achieved by minimizing $\max\left\{\lambda+\mu\epsilon,\nu+\frac{\mu}{4\epsilon}\right\}$, which is attained when $\lambda+\mu\epsilon=\nu+\frac{\mu}{4\epsilon}$. Solving for $\epsilon$ yields
\begin{align*}
\epsilon=\frac{(\nu-\lambda)+\sqrt{(\nu-\lambda)^2+\mu^2}}{2\mu}.
\end{align*}
For this choice, $\lambda+\mu\epsilon=\nu+\frac{\mu}{4\epsilon}=c^2$, where
\begin{align*}
c=\sqrt{\frac{\lambda+\nu+\sqrt{(\nu-\lambda)^2+\mu^2}}{2}}.
\end{align*}
Hence, the frame bounds simplify to $A\left(\frac{1-c}{1+c}\right) \ \text{and}\ B\left(\frac{1+c}{1-c}\right)$. Thus the optimized frame bounds are determined by the single parameter $c\in[0,1)$. Moreover, as $c\to 0$,
\begin{align*}
A\left(\frac{1-c}{1+c}\right) \to A \ \text{and}\ B\left(\frac{1+c}{1-c}\right) \to B,
\end{align*}
showing that the optimized frame bounds approach the original frame bounds. Hence, for sufficiently small perturbations, the perturbed HS-frame inherits frame bounds that are close to those of the original HS-frame.
\end{rem}

We provide the following example to illustrate Theorem \ref{gp2}.
\begin{exa}\label{exnew}
Let $\mathcal{H} = \mathcal{K} = \ell^2(\mathbb{N})$, and let $\{e_i\}_{i \in \mathbb{N}}$ be the canonical orthonormal basis of $\ell^2(\mathbb{N})$. Define $G_i:\ell^2(\mathbb{N}) \to  \mathcal{C}_2(\ell^2(\mathbb{N}))$ as
\begin{align*}
G_i(x)=\langle x, e_i \rangle e_i \otimes e_i, \ i \in \mathbb{N}, \ x \in \ell^2(\mathbb{N}).
\end{align*}
Then,
\begin{align*}
\sum_{i \in \mathbb{N}} \| G_i(x) \|_2^2 =\sum_{i \in \mathbb{N}} \| \langle x, e_i \rangle e_i \otimes e_i \|_2^2=\sum_{i \in \mathbb{N}} | \langle x, e_i \rangle |^2=\|x\|^2, \ x \in \ell^2(\mathbb{N}).
\end{align*}
This shows that $\{G_i\}_{i \in \mathbb{N}}$ is a Parseval HS-frame for $\ell^2(\mathbb{N})$ with frame bounds $A=B=1$. \\
Define $H_i:\ell^2(\mathbb{N}) \to  \mathcal{C}_2(\ell^2(\mathbb{N}))$ as
\begin{align*}
H_i(x) = \sqrt{2}\langle x, e_i \rangle \, (e_i \otimes e_i) + \langle x, e_i \rangle \, (e_{i+1} \otimes e_i), \ i \in \mathbb{N}, \ x \in \ell^2(\mathbb{N}).
\end{align*}
For every $A\in\mathcal{C}_2(\ell^2(\mathbb{N}))$ and $x\in\ell^2(\mathbb{N})$, we have
\begin{align*}
[G_i(x),A]_{\text{tr}}=[\langle x,e_i\rangle(e_i\otimes e_i),A]_{\text{tr}}=\langle x,e_i\rangle[e_i\otimes e_i,A]_{\text{tr}}=\left\langle x,\,[A,e_i\otimes e_i]_{\text{tr}} \ e_i\right\rangle.
\end{align*}
This implies that $G_i^*(A)=[A,e_i\otimes e_i]_{\text{tr}} \ e_i, \ A\in\mathcal{C}_2(\ell^2(\mathbb{N})).$
Consequently, for any $x\in\ell^2(\mathbb{N})$,
\begin{align*}
G_i^*G_i(x)&=[G_i(x),e_i\otimes e_i]_{\text{tr}} \ e_i=\langle x,e_i\rangle[e_i\otimes e_i,e_i\otimes e_i]_{\text{tr}} \ e_i=\langle x,e_i\rangle\langle e_i,e_i\rangle\langle e_i,e_i\rangle e_i=\langle x,e_i\rangle e_i.
\end{align*}
Similarly,
\begin{align*}
[H_i(x),A]_{\text{tr}}=\langle x,e_i\rangle[\sqrt{2}e_i\otimes e_i+e_{i+1}\otimes e_i, A]_{\text{tr}}=\big\langle x,[A,\sqrt{2}e_i\otimes e_i+e_{i+1}\otimes e_i]_{\text{tr}} \ e_i\big\rangle
\end{align*}
implies that
\begin{align*}
H_i^*(A)=[A,\sqrt{2}e_i\otimes e_i+e_{i+1}\otimes e_i]_{\text{tr}} \ e_i, \ A\in\mathcal{C}_2(\ell^2(\mathbb{N})).
\end{align*}
Therefore,
\begin{align*}
&H_i^*H_i(x)\\
&=[H_i(x),\sqrt{2}e_i\otimes e_i+e_{i+1}\otimes e_i]_{\text{tr}} \ e_i\\
&=\langle x,e_i\rangle[\sqrt{2}e_i\otimes e_i+e_{i+1}\otimes e_i,\sqrt{2}e_i\otimes e_i+e_{i+1}\otimes e_i]_{\text{tr}} \ e_i\\
&=\langle x,e_i\rangle\Big(2[e_i\otimes e_i,e_i\otimes e_i]_{\text{tr}}+\sqrt{2}[e_i\otimes e_i,e_{i+1}\otimes e_i]_{\text{tr}}+\sqrt{2}[e_{i+1}\otimes e_i,e_i\otimes e_i]_{\text{tr}}+[e_{i+1}\otimes e_i,e_{i+1}\otimes e_i]_{\text{tr}}\Big)e_i\\
&=\langle x,e_i\rangle\Big(2\langle e_i,e_i\rangle\langle e_i,e_i\rangle+\sqrt{2}\langle e_i,e_{i+1}\rangle\langle e_i,e_i\rangle+\sqrt{2}\langle e_{i+1},e_i\rangle\langle e_i,e_i\rangle+\langle e_{i+1},e_{i+1}\rangle\langle e_i,e_i\rangle\Big)e_i\\
&=3\langle x,e_i\rangle e_i, \ x\in\ell^2(\mathbb{N}).
\end{align*}
For every finite subset $J \subseteq \mathbb{N}$,
\begin{align*}
&\Big\|\sum_{i\in J}(G_i^*G_i-H_i^*H_i)x\Big\|^2\\
&=\Big\|\sum_{i\in J}\big(\langle x,e_i\rangle e_i-3\langle x,e_i\rangle e_i \big)\Big\|^2\\
&=4\sum_{i\in J}|\langle x,e_i\rangle|^2\\
&\leq \frac{1}{2}\sum_{i\in J}|\langle x,e_i\rangle|^2+\frac{1}{2}\Big(\sum_{i\in J}|\langle x,e_i\rangle|^2\Big)^{1/2}\times 3\Big(\sum_{i\in J}|\langle x,e_i\rangle|^2\Big)^{1/2}+\frac{1}{3}\times 9\sum_{i\in J}|\langle x,e_i\rangle|^2\\
&=\frac{1}{2}\Big\|\sum_{i\in J}G_i^*G_ix\Big\|^2+\frac{1}{2}\Big\|\sum_{i\in J}G_i^*G_ix\Big\|\Big\|\sum_{i\in J}H_i^*H_ix\Big\|+\frac{1}{3}\Big\|\sum_{i\in J}H_i^*H_ix\Big\|^2.
\end{align*}
The condition \eqref{gi} of Theorem \ref{gp2} is satisfied for $\lambda=\frac{1}{2}, \mu=\frac{1}{2}, \nu=\frac{1}{3}$. It can be verified that
\begin{align*}
\max\left\{\lambda+\mu\epsilon,\nu+\frac{\mu}{4\epsilon}\right\}=\max\left\{\frac{1}{2}+\frac{\epsilon}{2}, \ \frac{1}{3}+\frac{1}{8\epsilon}\right\}<1
\end{align*}
whenever $\frac{3}{16} < \epsilon < 1$. Hence, by Theorem \ref{gp2}, the family $\{H_i\}_{i\in\mathbb N}$ is a HS-frame for $\ell^2(\mathbb{N})$.
\end{exa}

We next establish a stability result based on perturbations of the pre-frame operator. While Theorem \ref{gp2} is formulated in terms of perturbations of the frame operator, the following theorem considers perturbations of the pre-frame operator and yields a different set of frame bounds.
\begin{thm}\label{th2}
Let $\{G_i \}_{i\in I}$ be an HS-frame for $\mathcal{H}$ with respect to $\mathcal{K}$ with bounds $A,B$, and let $\{H_i\}_{i\in I}$ be a family of operators from $\mathcal{H}$ into $\mathcal{C}_2(\mathcal{K})$. Let  $\lambda,\mu,\nu\geq0$ such that for any finite subset $J \subseteq I$,
\begin{align}\label{enew}
\Big\|\sum_{i\in J}(G_i^*A_i-H_i^*A_i)\Big\|^2&\leq\lambda\Big\|\sum_{i\in J}G_i^*A_i\Big\|^2
+\mu\Big\|\sum_{i\in J}G_i^*A_i\Big\|\Big\|\sum_{i\in J}H_i^*A_i\Big\|+\nu\Big\|\sum_{i\in J}H_i^*A_i\Big\|^2, \ A_i\in \mathcal{C}_2(\mathcal{K}).
\end{align}
If there exists an $\epsilon>0$ such that $\max\big\{\lambda+\mu\epsilon,\nu+\frac{\mu}{4\epsilon}\big\}<1$,
then $\{H_i\}_{i\in I}$ is an HS-frame for $\mathcal{H}$ with respect to $\mathcal{K}$ with bounds
$A\left(\frac{1-\sqrt{\lambda+\mu\epsilon}}{1+\sqrt{\nu+\frac{\mu}{4\epsilon}}}\right)^2$ and
$B\left(\frac{1+\sqrt{\lambda+\mu\epsilon}}{1-\sqrt{\nu+\frac{\mu}{4\epsilon}}}\right)^2$.
\end{thm}
\begin{proof}
For any finite subset $J\subseteq I$ and $A_i\in \mathcal{C}_2(\mathcal{K})$, we compute
\begin{align}\label{eq1}
&\Big\|\sum_{i\in J}(G_i^*A_i-H_i^*A_i)\Big\|^2\notag\\
& \leq \lambda\Big\|\sum_{i\in J}G_i^*A_i\Big\|^2+\mu\Big\|\sum_{i\in J}G_i^*A_i\Big\|\Big\|\sum_{i\in J}H_i^*A_i\Big\|+\nu\Big\|\sum_{i\in J}H_i^*A_i\Big\|^2\notag\\
& \leq \lambda\Big\|\sum_{i\in J}G_i^*A_i\Big\|^2+\mu\epsilon\Big\|\sum_{i\in J}G_i^*A_i\Big\|^2 +\frac{\mu}{4\epsilon} \Big\|\sum_{i\in J}H_i^*A_i\Big\|^2 +\nu\Big\|\sum_{i\in J}H_i^*A_i\Big\|^2 \quad \text{(by \eqref{y2})}\notag\\
&=\left(\sqrt{\lambda+\mu\epsilon} \ \Big\|\sum_{i\in J}G_i^*A_i\Big\|\right)^2+\left( \sqrt{\nu+\frac{\mu}{4\epsilon}} \ \Big\|\sum_{i\in J}H_i^*A_i\Big\|\right)^2\notag\\
& \leq \left(\sqrt{\lambda+\mu\epsilon} \ \Big\|\sum_{i\in J}G_i^*A_i\Big\| +\sqrt{\nu+\frac{\mu}{4\epsilon}} \ \Big\|\sum_{i\in J}H_i^*A_i\Big\|\right)^2\notag\\
\intertext{that gives}
&\Big\|\sum_{i\in J}(G_i^*A_i-H_i^*A_i)\Big\|  \leq \sqrt{\lambda+\mu\epsilon} \ \Big\|\sum_{i\in J}G_i^*A_i\Big\| +\sqrt{\nu+\frac{\mu}{4\epsilon}} \ \Big\|\sum_{i\in J}H_i^*A_i\Big\|.
\end{align}
Now
\begin{align*}
\Big\|\sum_{i\in J}G_i^*A_i\Big\| =\sup_{\|y\|=1}\Big|\Big\langle\sum_{i\in J}G_i^*A_i,y\Big\rangle\Big|
=\sup_{\|y\|=1}\Big|\sum_{i\in J}[ A_i,G_i y]_{\textbf{tr}}\Big|
\leq \Big(\sum_{i\in J}\|A_i\|_2^2\Big)^{1/2}\sup_{\|y\|=1}\Big(\sum_{i\in J}\|G_i y\|_2^2\Big)^{1/2}
\intertext{and}
\Big\|\sum_{i\in J}H_i^*A_i\Big\|\leq\Big\|\sum_{i\in J}(G_i^*A_i-H_i^*A_i)\Big\|+\Big\|\sum_{i\in J}G_i^*A_i\Big\|
\leq\Big(1+\sqrt{\lambda+\mu\epsilon}\Big)\Big\|\sum_{i\in J}G_i^*A_i\Big\|+\sqrt{\nu+\frac{\mu}{4\epsilon}}\Big\|\sum_{i\in J}H_i^*A_i\Big\|
\intertext{gives}
\Big\|\sum_{i\in J}H_i^*A_i\Big\|\leq\frac{1+\sqrt{\lambda+\mu\epsilon}}{1-\sqrt{\nu+\frac{\mu}{4\epsilon}}} \ \Big\|\sum_{i\in J}G_i^*A_i\Big\|\leq\sqrt{B}\left(\frac{1+\sqrt{\lambda+\mu\epsilon}}{1-\sqrt{\nu+\frac{\mu}{4\epsilon}}}\right)\Big(\sum_{i\in J}\|A_i\|_2^2\Big)^{1/2}.
\end{align*}
This shows that $\mathcal{T}_H:\bigoplus\limits_{i\in I} \mathcal{C}_2(\mathcal{K})\to \mathcal{H}$ defined by
\begin{align*}
\mathcal{T}_H\big(\{A_i\}_{i\in I}\big)=\sum_{i\in I}H_i^*A_i, \ \{A_i\}_{i\in I} \in \bigoplus\limits_{i \in I} \mathcal{C}_2(\mathcal{K})
\end{align*}
is a well-defined and bounded operator with $\|\mathcal{T}_H\| \leq \sqrt{B}\left(\frac{1+\sqrt{\lambda+\mu\epsilon}}{1-\sqrt{\nu+\frac{\mu}{4\epsilon}}}\right)$. For a fixed finite subset $J\subseteq I$ and $x\in\mathcal H$, define $\{A_i\}_{i\in I}\in \bigoplus\limits_{i\in I} \mathcal{C}_2(\mathcal{K})$ by
\begin{align*}
A_i =\begin{cases}
H_i x, & i \in J,\\
0, & i \notin J.
\end{cases}
\end{align*}
Then
\begin{align*}
\sum_{i\in J}\|H_i x\|_2^2=\sum_{i\in J}[H_i x,H_i x]_{\text{tr}}
&=\sum_{i\in J}\langle H_i^*H_i x,x\rangle \\
&=\Big\langle \sum_{i\in I}H_i^*A_i,x\Big\rangle\\
&\leq \Big\|\sum_{i\in I}H_i^*A_i \Big\|\,\|x\| \\
&\leq \|\mathcal{T}_H\|\,\|\{A_i\}_{i\in I}\|\,\|x\| \\
&=\|\mathcal{T}_H\|\Big(\sum_{i\in J}\|H_i x\|_2^2\Big)^{1/2}\|x\|.
\end{align*}
This implies that $\sum\limits_{i\in J}\|H_i x\|_2^2\leq\|\mathcal{T}_H\|^2\|x\|^2$. Since $J$ is arbitrary, it follows that
\begin{align*}
\sum_{i\in I}\|H_i x\|_2^2\leq\|\mathcal{T}_H\|^2\|x\|^2, \ x \in \mathcal{H}.
\end{align*}
Therefore, $\{H_i\}_{i\in I}$ is an HS-Bessel sequence for $\mathcal{H}$ with Bessel bound $B\left(\frac{1+\sqrt{\lambda+\mu\epsilon}}{1-\sqrt{\nu+\frac{\mu}{4\epsilon}}}\right)^2$.

Now, let $\mathcal{T}_G$ and $\mathcal{S}_G$ be the pre-frame operator and the frame operator associated with $\{G_i\}_{i\in I}$. Define $\mathcal{E}=\mathcal{T}_H\mathcal{T}_G^*{\mathcal{S}_G}^{-1}$. Applying \eqref{eq1} to finite subsets $J \subseteq I$ with $A_i=G_i {\mathcal{S}_G}^{-1}x$ and passing to the limit over finite subsets of $I$, we obtain
\begin{align*}
\|x-\mathcal{E}x\|&=\Big\|\sum_{i\in I}\big(G_i^*G_i {\mathcal{S}_G}^{-1}x-H_i^*G_i {\mathcal{S}_G}^{-1}x\big)\Big\| \\
&\leq\sqrt{\lambda+\mu\epsilon} \ \Big\|\sum_{i\in I}G_i^*G_i {\mathcal{S}_G}^{-1}x\Big\|
+\sqrt{\nu+\frac{\mu}{4\epsilon}} \ \Big\|\sum_{i\in I}H_i^*G_i {\mathcal{S}_G}^{-1}x\Big\|\\
&=\sqrt{\lambda+\mu\epsilon} \ \|x\|+\sqrt{\nu+\frac{\mu}{4\epsilon}} \ \|\mathcal{E}x\|\\
\end{align*}
Since $0\leq\max\Big\{\sqrt{\lambda+\mu\epsilon},\sqrt{\nu+\frac{\mu}{4\epsilon}}\Big\}<1,
$ Lemma \ref{L1} implies that $\mathcal{E}$ is invertible and $\|\mathcal{E}^{-1}\|\leq\frac{1+\sqrt{\nu+\frac{\mu}{4\epsilon}}}{1-\sqrt{\lambda+\mu\epsilon}}$. Next,
\begin{align*}
\|x\|^2=|\langle \mathcal{E}\mathcal{E}^{-1}x,x\rangle|
&=\Big|\Big\langle\sum_{i\in I}H_i^*G_i {\mathcal{S}_G}^{-1}\mathcal{E}^{-1}x,x\Big\rangle\Big| \\
&=\Big|\sum_{i\in I}[G_i {\mathcal{S}_G}^{-1}\mathcal{E}^{-1}x, H_i x]_{\textbf{tr}}\Big| \\
&\leq \sum_{i\in I}\|G_i {\mathcal{S}_G}^{-1}\mathcal{E}^{-1}x\|_2 \|H_i x\|_2 \\
&\leq\Big(\sum_{i\in I}\|G_i {\mathcal{S}_G}^{-1}\mathcal{E}^{-1}x\|_2^2\Big)^{1/2}\Big(\sum_{i\in I}\|H_i x\|_2^2\Big)^{1/2}\\
& \leq\frac{1}{\sqrt{A}}\|\mathcal{E}^{-1}x\|\Big(\sum_{i\in I}\|H_i x\|_2^2\Big)^{1/2}\\
& \leq \frac{1}{\sqrt{A}} \left(\frac{1+\sqrt{\nu+\frac{\mu}{4\epsilon}}}{1-\sqrt{\lambda+\mu\epsilon}}\right)\|x\|\Big(\sum_{i\in I}\|H_i x\|_2^2\Big)^{1/2}\\
\intertext{that implies}
&\sum_{i\in I}\|H_i x\|_2^2\geq A\left(\frac{1-\sqrt{\lambda+\mu\epsilon}}{1+\sqrt{\nu+\frac{\mu}{4\epsilon}}}\right)^2\|x\|^2, \ x \in \mathcal{H}.
\end{align*}
Hence, $\{H_i\}_{i\in I}$ is an HS-frame for $\mathcal{H}$ with respect to $\mathcal{K}$ with the required frame bounds.
\end{proof}

\begin{rem}
The perturbation criterion in Theorem \ref{th2} is sufficient but not necessary for a family to be an HS-frame. Furthermore, Theorem \ref{gp2} and Theorem \ref{th2} are not equivalent. Indeed, the family $\{H_i\}_{i\in\mathbb{N}}$ considered in Example \ref{exnew} satisfies the hypotheses of Theorem \ref{gp2}, and hence is an HS-frame, whereas it does not satisfy the perturbation condition of Theorem \ref{th2}. To verify the latter, let $J=\{1\}$ and choose $A_1=e_1\otimes e_1-\sqrt{2}e_2\otimes e_1\in\mathcal{C}_2(\ell^2(\mathbb{N}))$. Then,
\begin{align*}
&G_1^*A_1=[A_1,e_1\otimes e_1]_{\textbf{tr}}\ e_1=[e_1\otimes e_1 ,e_1\otimes e_1]_{\textbf{tr}}\ e_1-\sqrt{2}[e_2\otimes e_1,e_1\otimes e_1]_{\textbf{tr}}\ e_1=e_1, \\
&H_1^*A_1=[A_1,\sqrt{2}e_1\otimes e_1+e_2\otimes e_1]_{\textbf{tr}}\ e_1=[e_1\otimes e_1-\sqrt{2}e_2\otimes e_1,\sqrt{2}e_1\otimes e_1+e_2\otimes e_1]_{\textbf{tr}}\ e_1=0.
\end{align*}
If the inequality \eqref{enew} in Theorem \ref{th2} were satisfied, then
\begin{align*}
1=\|G_1^*A_1-H_1^*A_1\|^2\le\lambda\|G_1^*A_1\|^2+\mu\|G_1^*A_1\|\|H_1^*A_1\|+\nu\|H_1^*A_1\|^2=\lambda.
\end{align*}
However $\mu\ge0$ and $\max\left\{\lambda+\mu\epsilon,\nu+\frac{\mu}{4\epsilon}\right\}<1$ imply $\lambda \leq \lambda+\mu\epsilon<1$, which is a contradiction.
\end{rem}

We now present an example illustrating the applicability of Theorem \ref{th2}.

\begin{exa}
Let $\mathcal{H} = \mathcal{K} = \ell^2(\mathbb{N})$, and $\{G_i\}_{i \in \mathbb{N}}$ be the Parseval HS-frame given in Example \ref{exnew}. Let $\Theta \in \mathfrak{B}(\ell^2(\mathbb{N}))$ such that $\|I-\Theta\|<1$. Define $H_i:\ell^2(\mathbb{N}) \to  \mathcal{C}_2(\ell^2(\mathbb{N}))$ as
\begin{align*}
H_i=G_i \Theta, \ i \in \mathbb{N}.
\end{align*}
For every finite subset $J\subseteq \mathbb{N}$ and $A_i\in\mathcal{C}_2(\ell^2(\mathbb{N}))$, we have
\begin{align*}
\Big\|\sum_{i\in J}(G_i^*A_i-H_i^*A_i)\Big\|^2=\Big\|\sum_{i\in J}(G_i^*A_i-\Theta^*G_i^*A_i)\Big\|^2=\Big\|(I-\Theta^*)\sum_{i\in J}G_i^*A_i\Big\|^2\leq \|I-\Theta\|^2\Big\|\sum_{i\in J}G_i^*A_i\Big\|^2.
\end{align*}
The condition \eqref{enew} holds with $\lambda=\|I-\Theta\|^2,\mu=0,\nu=0$.
Since $\max\left\{\lambda+\mu\epsilon,\nu+\frac{\mu}{4\epsilon}\right\}=\|I-\Theta\|^2<1$, by Theorem \ref{th2}, $\{H_i\}_{i\in\mathbb{N}}$ is an HS-frame for $\ell^2(\mathbb{N})$ with frame bounds $\big(1-\|I-\Theta\|\big)^2$ and $\big(1+\|I-\Theta\|\big)^2$.
\end{exa}

\section{Pollard-Hilding-Type Stability of HS-Frames}\label{secp}

We now turn to nonlinear $k$-power perturbation results of Pollard-Hilding-type for HS-frames. In contrast to the quadratic Nagy-type perturbation framework developed in Section \ref{seci}, the present approach is based on nonlinear $k$-power inequalities, leading to a distinct class of stability criteria. As before, our objective is to derive explicit sufficient conditions under which the perturbed family remains an HS-frame, together with computable frame bounds.

\begin{thm}\label{gp3}
Let $\{G_i\}_{i\in I}$ be an HS-frame for $\mathcal{H}$ with respect to $\mathcal{K}$ with bounds $A,B$, and let $\{H_i\}_{i\in I}$ be a family of operators from $\mathcal{H}$ into $\mathcal{C}_2(\mathcal{K})$. If for each $k>0$, there exist constants $\lambda_1,\lambda_2\ge 0$ satisfying $\max\{\lambda_1,\lambda_2\}<\min\{1,2^{k-1}\}$
such that for every finite subset $J\subseteq I$ and $x\in\mathcal{H}$,
\begin{align}\label{eqp1}
\Big\|\sum_{i\in J}(G_i^*G_i x-H_i^*H_i x)\Big\|^k\leq \lambda_1\Big\|\sum_{i\in J}G_i^*G_i x\Big\|^k
+\lambda_2\Big\|\sum_{i\in J}H_i^*H_i x\Big\|^k.
\end{align}
Then, $\{H_i\}_{i\in I}$ is an HS-frame for $\mathcal{H}$ with respect to $\mathcal{K}$ with bounds
$A\left(\frac{1-c^{1/k}\lambda_1^{1/k}}{1+c^{1/k}\lambda_2^{1/k}}\right)$
and $B\left(\frac{1+c^{1/k}\lambda_1^{1/k}}{1-c^{1/k}\lambda_2^{1/k}}\right)$, where
$c=\begin{cases}
1, & k\ge 1,\\[1mm]
2^{1-k}, & k\le 1.
\end{cases}$
\end{thm}
\begin{proof}
For any finite set $J\subseteq I$ and $x \in \mathcal{H}$, we compute
\begin{align*}
&\Big\|\sum_{i\in J}(G_i^*G_i x-H_i^*H_i x)\Big\|\\
&\leq \left[\lambda_1\Big\|\sum_{i\in J}G_i^*G_i x\Big\|^k+\lambda_2\Big\|\sum_{i\in J}H_i^*H_i x\Big\|^k \right]^{1/k} \\
&=\left[\left(\lambda_1^{1/k}\Big\|\sum_{i\in J}G_i^*G_i x\Big\|\right)^k+\left(\lambda_2^{1/k}\Big\|\sum_{i\in J}H_i^*H_i x\Big\|\right)^k \right]^{1/k}\\
&\leq c^{1/k}\left[\left(\lambda_1^{1/k}\Big\|\sum_{i\in J}G_i^*G_i x\Big\|+\lambda_2^{1/k}\Big\|\sum_{i\in J}H_i^*H_i x\Big\|\right)^k \right]^{1/k} \ \text{(using Theorem \ref{Theorem 2.8})}\\
&=c^{1/k}\lambda_1^{1/k}\Big\|\sum_{i\in J}G_i^*G_i x\Big\|+c^{1/k}\lambda_2^{1/k}\Big\|\sum_{i\in J}H_i^*H_i x\Big\|, \quad \text{where} \ c=\begin{cases}
1,  & k \geq 1,\\
2^{1-k}, & k \leq 1.
\end{cases}
\end{align*}
Since $\max\{\lambda_1,\lambda_2\} < \min\{1,2^{k-1}\}$, it follows that $\max\big\{c^{1/k} \lambda_1^{1/k} ,c^{1/k} \lambda_2^{1/k} \big\} < 1$. Hence, by following the same argument as in the proof of Theorem \ref{gp2}, it follows that $\{H_i\}_{i\in I}$ is an HS-frame for $\mathcal{H}$ with frame bounds
$A\left(\frac{1-c^{1/k}\lambda_1^{1/k}}{1+c^{1/k}\lambda_2^{1/k}}\right)$
and $B\left(\frac{1+c^{1/k}\lambda_1^{1/k}}{1-c^{1/k}\lambda_2^{1/k}}\right)$.
\end{proof}

\begin{rem}
The perturbation estimate used in the proof of Theorem \ref{gp2} is first reduced, via Young's inequality, to a quadratic inequality involving separate contributions from the original and perturbed families. Theorem \ref{gp3}, on the other hand, is based on a direct nonlinear $k$-power perturbation estimate and therefore avoids this quadratic reduction. Thus, Theorem \ref{gp3} provides a natural nonlinear counterpart to the quadratic framework underlying Theorem \ref{gp2}.
\end{rem}

The following example illustrates Theorem \ref{gp3}.

\begin{exa}
Let $\mathcal{H}=\mathcal{K}=\ell^{2}(\mathbb{N})$. Consider the Parseval HS-frame $\{G_i \}_{i\in \mathbb{N}}$ and the family of operators $\{H_i \}_{i\in \mathbb{N}}$ given in Example \ref{exnew}. For every finite subset $J\subseteq\mathbb{N}$ and every $x\in\mathcal{H}$, we have
\begin{align*}
\Big\|\sum_{i\in J}G_i^{*}G_i x\Big\|^{k}&=\Big(\sum_{i\in J}|\langle x, e_i \rangle|^2\Big)^{k/2},\\
\Big\|\sum_{i\in J}H_i^{*}H_i x\Big\|^{k}&=3^{k}\Big(\sum_{i\in J}|\langle x, e_i \rangle|^2\Big)^{k/2},\\
\Big\|\sum_{i\in J}(G_i^{*}G_ix-H_i^{*}H_ix)\Big\|^{k}&=2^{k}\Big(\sum_{i\in J}|\langle x, e_i \rangle|^2\Big)^{k/2}.
\end{align*}
\textbf{Case 1:} For $k\ge1$, choose $\lambda_1=\lambda_2=\alpha_k$ where $\frac{2^{k}}{1+3^{k}}<\alpha_k<1$.
Then,
\begin{align*}
\Big\|\sum_{i\in J}(G_i^{*}G_ix-H_i^{*}H_ix)\Big\|^{k}&=2^{k}\Big(\sum_{i\in J}|\langle x, e_i \rangle|^2\Big)^{k/2}\\
&\leq \alpha_k\Big(\sum_{i\in J}|\langle x, e_i \rangle|^2\Big)^{k/2}+\alpha_k3^{k}\Big(\sum_{i\in J}|\langle x, e_i \rangle|^2\Big)^{k/2}\\
&=\lambda_1\Big\|\sum_{i\in J}G_i^{*}G_i x\Big\|^{k}+\lambda_2\Big\|\sum_{i\in J}H_i^{*}H_i x\Big\|^{k}.
\end{align*}
Moreover, $\max\{\lambda_1,\lambda_2\}=\alpha_k<1=\min\{1,2^{k-1}\}$. Therefore, by Theorem \ref{gp3}, $\{H_i\}_{i\in\mathbb{N}}$ is an HS-frame for $\ell^{2}(\mathbb{N})$ with frame bounds $\left(\frac{1-\alpha_k^{1/k}}{1+\alpha_k^{1/k}}\right)$ and $\left(\frac{1+\alpha_k^{1/k}}{1-\alpha_k^{1/k}}\right)$.

\noindent\textbf{Case 2:} For $0<k<1$, choose $\lambda_1=\lambda_2=\beta_k2^{k-1}$ where $\frac{2}{1+3^{k}}<\beta_k<1$.
In this case,
\begin{align*}
\Big\|\sum_{i\in J}(G_i^{*}G_ix-H_i^{*}H_ix)\Big\|^{k}&=2^{k}\Big(\sum_{i\in J}|\langle x, e_i \rangle|^2\Big)^{k/2}\\
&\leq \beta_k2^{k-1}\Big(\sum_{i\in J}|\langle x, e_i \rangle|^2\Big)^{k/2}+\beta_k2^{k-1}3^{k}\Big(\sum_{i\in J}|\langle x, e_i \rangle|^2\Big)^{k/2}\\
&=\lambda_1\Big\|\sum_{i\in J}G_i^{*}G_i x\Big\|^{k}+\lambda_2\Big\|\sum_{i\in J}H_i^{*}H_i x\Big\|^{k}.
\end{align*}
Furthermore, $\max\{\lambda_1,\lambda_2\}=\beta_k2^{k-1}<2^{k-1}=\min\{1,2^{k-1}\}$. Hence, by Theorem \ref{gp3}, $\{H_i\}_{i\in\mathbb{N}}$ is an HS-frame for $\ell^{2}(\mathbb{N})$ with frame bounds $\left(\frac{1-\beta_k^{1/k}}{1+\beta_k^{1/k}}\right)$ and $\left(\frac{1+\beta_k^{1/k}}{1-\beta_k^{1/k}}\right)$.
\end{exa}

Just as the quadratic Nagy-type framework admits both frame operator and preframe operator formulations, the Pollard-Hilding-type framework also extends naturally to perturbations of the preframe operator. This leads to the following stability theorem with explicit frame bounds.

\begin{thm}\label{th3}
Let $\{G_i \}_{i\in I}$ be an HS-frame for $\mathcal{H}$ with respect to $\mathcal{K}$ with bounds $A,B$, and let $\{H_i\}_{i\in I}$ be a family of operators from $\mathcal{H}$ into $\mathcal{C}_2(\mathcal{K})$. If for each $k > 0$, there exist constants $\lambda_{1}, \lambda_{2} \geq 0$ satisfying $\max\{\lambda_{1}, \lambda_{2}\} < \min\{1, 2^{k-1}\}$ such that for all finite subsets $J \subseteq I$,
\begin{align}\label{eqp2}
\Big\|\sum_{i\in J}(G_i^*A_i-H_i^*A_i)\Big\|^k&\leq \lambda_1\Big\|\sum_{i\in J}G_i^*A_i\Big\|^k
+\lambda_2\Big\|\sum_{i\in J}H_i^*A_i\Big\|^k , \ A_i\in \mathcal{C}_2(\mathcal{K}).
\end{align}
Then, $\{H_i\}_{i\in I}$ is an HS-frame for $\mathcal{H}$ with respect to $\mathcal{K}$ with bounds
$A\left(\frac{1-c^{1/k} \lambda_1^{1/k} }{1+c^{1/k} \lambda_2^{1/k} }\right)^2$ and
$B\left(\frac{1+c^{1/k} \lambda_1^{1/k} }{1-c^{1/k} \lambda_2^{1/k} }\right)^2$, where $c=\begin{cases}
1,  & k \geq 1,\\
2^{1-k}, & k \leq 1 .
\end{cases}$
\end{thm}
\begin{proof}
Arguing as in the proof of Theorem \ref{gp3}, we obtain
\begin{align*}
\Big\|\sum_{i\in J}(G_i^*A_i-H_i^*A_i)\Big\|
\leq c^{1/k} \lambda_1^{1/k} \Big\|\sum_{i\in J}G_i^*A_i\Big\|+c^{1/k} \lambda_2^{1/k} \Big\|\sum_{i\in J}H_i^*A_i\Big\|, \quad \text{where} \ c=\begin{cases}
1,  & k \geq 1,\\
2^{1-k}, & k \leq 1.
\end{cases}
\end{align*}
Since $\max\{\lambda_1,\lambda_2\} < \min\{1,2^{k-1}\}$, we have $\max\big\{c^{1/k} \lambda_1^{1/k} ,c^{1/k} \lambda_2^{1/k} \big\} < 1$. The remainder of the proof proceeds as in the proof of Theorem \ref{th2}.
\end{proof}

Next, we give an illustration of Theorem \ref{th3}.
\begin{exa}
Let $\mathcal{H}=\mathcal{K}=\ell^2(\mathbb{N})$, and let $\{G_i\}_{i\in\mathbb{N}}$ be the Parseval HS-frame given in Example \ref{exnew}. Let $\Phi\in\mathfrak{B}(\ell^2(\mathbb{N}))$ satisfy $\|I-\Phi\|\leq \frac{1}{4}$. Define $H_i:\ell^2(\mathbb{N}) \to  \mathcal{C}_2(\ell^2(\mathbb{N}))$ as
\begin{align*}
H_i=G_i \Phi,\qquad i\in\mathbb{N}.
\end{align*}
For every finite subset $J\subseteq\mathbb{N}$ and every $A_i\in\mathcal{C}_2(\ell^2(\mathbb{N}))$, we have
\begin{align*}
\Big\|\sum_{i\in J}(G_i^*A_i-H_i^*A_i)\Big\|^k=\Big\|(I-\Phi^*)\sum_{i\in J}G_i^*A_i\Big\|^k\leq\left(\frac{1}{4}\right)^k\Big\|\sum_{i\in J}G_i^*A_i\Big\|^k.
\end{align*}
Moreover,
\begin{align*}
\Big\|\sum_{i\in J}H_i^*A_i\Big\|=\Big\|\Phi^*\sum_{i\in J}G_i^*A_i\Big\|\geq\big(1-\|I-\Phi\|\big)\Big\|\sum_{i\in J}G_i^*A_i\Big\|\geq
\frac{3}{4}\Big\|\sum_{i\in J}G_i^*A_i\Big\|.
\end{align*}
Since $2\left(\frac{1}{4}\right)^k \leq 1+\left(\frac{3}{4}\right)^k, k >0$, we obtain
\begin{align*}
\Big\|\sum_{i\in J}(G_i^*A_i-H_i^*A_i)\Big\|^k
\leq\frac{1}{2}\Big\|\sum_{i\in J}G_i^*A_i\Big\|^k+\frac{1}{2}\left(\frac{3}{4}\right)^k\Big\|\sum_{i\in J}G_i^*A_i\Big\|^k
\leq\frac{1}{2}\Big\|\sum_{i\in J}G_i^*A_i\Big\|^k+\frac{1}{2}\Big\|\sum_{i\in J}H_i^*A_i\Big\|^k.
\end{align*}
The condition \eqref{eqp2} is satisfied for $\lambda_1=\lambda_2=\frac{1}{2}$. Further, $\max\{\lambda_1,\lambda_2\}=\frac{1}{2}<\min\{1,2^{k-1}\}, \ k>0$. Therefore, by Theorem \ref{th3}, $\{H_i\}_{i\in\mathbb{N}}$ is an HS-frame for $\ell^2(\mathbb{N})$. Moreover, its frame bounds simplify to
$\left(\frac{1-2^{-1/k}}{1+2^{-1/k}}\right)^2$ and $\left(\frac{1+2^{-1/k}}{1-2^{-1/k}}\right)^2$ for $k \geq 1$, whereas they reduce to $\frac{1}{9}$ and $9$ for $0<k<1$.
\end{exa}

This completes the Pollard-Hilding-type stability analysis in both the frame and pre-frame operator settings. We conclude by considering a nonlinear perturbation condition involving mixed-power estimates of the original and perturbed frame operators, yielding an alternative stability criterion for HS-frames. For completeness, we also present the corresponding pre-frame operator version. Its proof is omitted, as it follows by arguments analogous to those developed in the preceding perturbation results.

\begin{thm}\label{pq}
Let $\{G_i\}_{i\in I}$ be an HS-frame for $\mathcal{H}$ with respect to $\mathcal{K}$ with bounds $A,B$, and let $\{H_i\}_{i\in I}$ be a family of operators from $\mathcal{H}$ into $\mathcal{C}_2(\mathcal{K})$. Assume there exist constants $\alpha, \beta \in (0,1)$ satisfying $\alpha+\beta=1$, and constants $\lambda_1,\lambda_2 \ge 0$ satisfying $\max\{\lambda_1,\lambda_2\}<1$,
such that for every finite subset $J\subseteq I$ and for all $x\in\mathcal{H}$,
\begin{align}\label{eqp3}
\Big\|\sum_{i\in J}(G_i^*G_i x - H_i^*H_i x)\Big\|\leq\lambda_1 \Big\|\sum_{i\in J}G_i^*G_i x\Big\|^\alpha\Big\|\sum_{i\in J}H_i^*H_i x\Big\|^\beta
+\lambda_2 \Big\|\sum_{i\in J}G_i^*G_i x\Big\|^\beta\Big\|\sum_{i\in J}H_i^*H_i x\Big\|^\alpha.
\end{align}
Then, $\{H_i\}_{i\in I}$ is an HS-frame for $\mathcal{H}$ with respect to $\mathcal{K}$ with  bounds
$A \left( \frac{1-\alpha\lambda_1 - \beta\lambda_2}{1+\beta\lambda_1 + \alpha\lambda_2}\right)$ and
$B \left( \frac{1+\alpha\lambda_1 +\beta\lambda_2 }{1-\beta\lambda_1 -\alpha\lambda_2 }\right)$.
\end{thm}
\begin{proof}
For any finite subset $J\subseteq I$ and $x \in \mathcal{H}$, we have
\begin{align*}
&\Big\|\sum_{i\in J}(G_i^*G_i x-H_i^*H_i x)\Big\|\\
&\leq \lambda_1 \Big\|\sum_{i\in J}G_i^*G_i x\Big\|^\alpha\Big\|\sum_{i\in J}H_i^*H_i x\Big\|^\beta+\lambda_2 \Big\|\sum_{i\in J}G_i^*G_i x\Big\|^\beta
\Big\|\sum_{i\in J}H_i^*H_i x\Big\|^\alpha.
\end{align*}
Applying Theorem \ref{Young} with $p=\frac{1}{\alpha}$ and $q=\frac{1}{\beta}$, we obtain
\begin{align*}
&\Big\|\sum_{i\in J}(G_i^*G_i x-H_i^*H_i x)\Big\|\\
&\leq \lambda_1\left(\alpha\Big\|\sum_{i\in J}G_i^*G_i x\Big\|+\beta\Big\|\sum_{i\in J}H_i^*H_i x\Big\|\right)+ \lambda_2\left(\beta\Big\|\sum_{i\in J}G_i^*G_i x\Big\|+\alpha\Big\|\sum_{i\in J}H_i^*H_i x\Big\|\right) \\
&=\big(\alpha\lambda_1 +\beta\lambda_2 \big)\Big\|\sum_{i\in J}G_i^*G_i x\Big\|+\big(\beta\lambda_1 +\alpha\lambda_2 \big)\Big\|\sum_{i\in J}H_i^*H_i x\Big\|.
\end{align*}
Since $\alpha+\beta=1$, the quantities $\alpha\lambda_1 +\beta\lambda_2 $ and $\beta\lambda_1 +\alpha\lambda_2 $
are convex combinations of $\lambda_1$ and $\lambda_2$. Therefore,
\begin{align*}
\max\big\{\alpha\lambda_1 +\beta\lambda_2 ,\beta\lambda_1 +\alpha\lambda_2 \big\}\leq \max\{\lambda_1,\lambda_2\} < 1.
\end{align*}
Consequently, the same argument as in the proof of Theorem \ref{gp2} implies that $\{H_i\}_{i\in I}$ is an HS-frame for $\mathcal{H}$ with respect to $\mathcal{K}$ with the desired frame bounds.
\end{proof}

\begin{thm}\label{pq1}
Let $\{G_i\}_{i\in I}$ be an HS-frame for $\mathcal{H}$ with respect to $\mathcal{K}$ with bounds $A,B$, and let $\{H_i\}_{i\in I}$ be a family of operators from $\mathcal{H}$ into $\mathcal{C}_2(\mathcal{K})$. Assume there exist constants $\alpha, \beta \in (0,1)$ satisfying $\alpha+\beta=1$ and constants $\lambda_1,\lambda_2 \ge 0$ satisfying $\max\{\lambda_1,\lambda_2\}<1$,
such that for every finite subset $J\subseteq I$,
\begin{align*}
\Big\|\sum_{i\in J}(G_i^*A_i - H_i^*A_i )\Big\|\leq\lambda_1 \Big\|\sum_{i\in J}G_i^*A_i \Big\|^\alpha\Big\|\sum_{i\in J}H_i^*A_i \Big\|^\beta
+\lambda_2 \Big\|\sum_{i\in J}G_i^*A_i \Big\|^\beta\Big\|\sum_{i\in J}H_i^*A_i \Big\|^\alpha,  \ A_i \in \mathcal{C}_2(\mathcal{K}).
\end{align*}
Then, $\{H_i\}_{i\in I}$ is an HS-frame for $\mathcal{H}$ with respect to $\mathcal{K}$ with  bounds
$A \left( \frac{1-\alpha\lambda_1 - \beta\lambda_2}{1+\beta\lambda_1 + \alpha\lambda_2}\right)^2$ and
$B \left( \frac{1+\alpha\lambda_1 +\beta\lambda_2 }{1-\beta\lambda_1 -\alpha\lambda_2 }\right)^2$.
\end{thm}

We end the paper with the following example illustrating Theorem \ref{pq}.

\begin{exa}
We consider $\{G_i \}_{i\in \mathbb{N}}$ and $\{H_i \}_{i\in \mathbb{N}}$ from Example \ref{exnew}. For every finite subset $J\subseteq\mathbb{N}$ and every $x\in\mathcal{H}$, we have
\begin{align*}
&\lambda_1\Big\|\sum_{i\in J}G_i^*G_i x\Big\|^\alpha\Big\|\sum_{i\in J}H_i^*H_i x\Big\|^\beta
+\lambda_2\Big\|\sum_{i\in J}G_i^*G_i x\Big\|^\beta\Big\|\sum_{i\in J}H_i^*H_i x\Big\|^\alpha\\
&=\lambda_1\left(\Big(\sum_{i\in J}|\langle x, e_i \rangle|^2\Big)^{1/2}\right)^\alpha \left(3\Big(\sum_{i\in J}|\langle x, e_i \rangle|^2\Big)^{1/2}\right)^\beta
+\lambda_2\left(\Big(\sum_{i\in J}|\langle x, e_i \rangle|^2\Big)^{1/2}\right)^\beta \left(3\Big(\sum_{i\in J}|\langle x, e_i \rangle|^2\Big)^{1/2}\right)^\alpha\\
&=\big(\lambda_1 3^\beta+ \lambda_2 3^\alpha \big)\Big(\sum_{i\in J}|\langle x, e_i \rangle|^2\Big)^{1/2}.
\end{align*}
Now choose $\alpha=\frac{1}{3}, \ \beta=\frac{2}{3}, \ \lambda_1=\frac{4}{5}$ and $\lambda_2=\frac{1}{3}$. Then, $\alpha+\beta=1$ and $\max\{\lambda_1, \lambda_2\}<1$.
Moreover,
\begin{align*}
\Big\|\sum_{i\in J}(G_i^*G_i-H_i^*H_i)x\Big\|&=2\Big(\sum_{i\in J}|\langle x,e_i\rangle|^2\Big)^{1/2}\\
&<\left(\frac{4}{5} \times 3^{2/3}+\frac{1}{3}\times 3^{1/3}\right)\Big(\sum_{i\in J}|\langle x,e_i\rangle|^2\Big)^{1/2}\\
&=\frac{4}{5}\Big\|\sum_{i\in J}G_i^*G_i x\Big\|^{1/3}\Big\|\sum_{i\in J}H_i^*H_i x\Big\|^{2/3}+\frac{1}{3}\Big\|\sum_{i\in J}G_i^*G_i x\Big\|^{2/3}
\Big\|\sum_{i\in J}H_i^*H_i x\Big\|^{1/3}.
\end{align*}
Therefore, all the hypotheses of Theorem \ref{pq} are satisfied and hence $\{H_i\}_{i\in\mathbb N}$ is an HS-frame.
\end{exa}

%$$\text{\textbf{Statements $\&$ Declarations}}$$
%There is no competing interests and no data is used in this study.

\mbox{}
\end{document}